\documentclass[reqno,a4paper]{amsart}
\usepackage{eucal,amsfonts,amssymb,amsmath,amsthm,epsfig,mathrsfs}

\usepackage{amscd,amsxtra}
\usepackage{enumerate}
\usepackage{latexsym}

\allowdisplaybreaks

 \makeatletter \@addtoreset{equation}{section}

\makeatother \makeatletter

\newtheorem{thm}{Theorem}[section]
\newtheorem{hyp}[thm]{Hypotheses}{\rm}
{\rm}
\newtheorem{lemm}[thm]{Lemma}

\newtheorem{prop}[thm]{Proposition}
\newtheorem{defi}[thm]{Definition}
\newtheorem{rmk}[thm]{Remark}{\rm}
\newtheorem{example}[thm]{Example}

\newcommand{\R}{{\mathbb R}}
\newcommand{\N}{{\mathbb N}}

\newcommand{\Rd}{\mathbb R^d}

\newcommand{\A}{\mathcal{A}}

\newcommand{\supp}{{\rm{supp}}\,}

\newcommand{\bd}{\begin{defi}}
\newcommand{\ed}{\end{defi}}
\newcommand{\nnm}{\nonumber}
\newcommand{\be}{\begin{equation}}
\newcommand{\ee}{\end{equation}}
\newcommand{\barr}{\begin{array}}
\newcommand{\earr}{\end{array}}
\newcommand{\bmn}{\begin{eqnarray}}
\newcommand{\emn}{\end{eqnarray}}
\newcommand{\bnm}{\begin{eqnarray*}}
\newcommand{\enm}{\end{eqnarray*}}
\newcommand{\bln}{\begin{subequations}}
\newcommand{\eln}{\end{subequations}}

\newcommand{\ba}{\begin{align}}
\newcommand{\ea}{\end{align}}
\newcommand{\banm}{\begin{align*}}
\newcommand{\eanm}{\end{align*}}

\title[On the gradient estimates for evolution operators]{On the gradient estimates for evolution operators associated to Kolmogorov operators}
\author[L. Angiuli]{Luciana Angiuli}
\address{Dipartimento di Matematica e Informatica, Universit\`a degli Studi di Parma, Parco Area delle Scienze 53/A, I-43124 Parma, Italy.}
\email{luciana.angiuli@unipr.it}

\date{\today}
\begin{document}

\begin{abstract}
We determine sufficient conditions for the occurrence of a pointwise gradient estimate 
for the evolution operators associated to nonautonomous second order parabolic operators 
with (possibly) unbounded coefficients.
Moreover we exhibit a class of operators which satisfy our conditions.
\end{abstract}

\maketitle
\section{Introduction}
Let $I$ be an open right halfline and let $\{\A(t)\}_{t \in I}$ be a family of second order differential operators defined on smooth functions $\zeta$ by
\begin{equation}\label{a(t)}
(\mathcal{A}(t)\zeta)(x)= {\textrm{Tr}}(Q(t,x)D^2\zeta(x))+ \langle b(t,x), \nabla \zeta(x)\rangle,
\end{equation}
where the (possibly) unbounded coefficients $Q=[q_{ij}]_{i, j=1, \ldots, d}$ and $b = (b_1, \ldots, b_d)$ are defined in $I\times\Rd$. Let us consider the nonautonomous Cauchy problem
\begin{equation}\label{p_e_intro}
\left\{
\begin{array}{ll}
D_tu(t,x)={\mathcal{A}}(t)u(t,x),\quad\quad & (t,x)\in (s,+\infty)\times \Rd,\\[1mm]
u(s,x)= f(x),\quad\quad & x\in \Rd,
\end{array}\right.
\end{equation}
with $s\in I$ and $f\in C_b(\Rd)$.
\noindent
In the pioneering paper \cite{KunLorLun09Non}, under suitable assumptions on the coefficients $q_{ij}$ and $b_{i}$, the authors prove the
wellposedness of the problem \eqref{p_e_intro} in the space of continuous and bounded functions defined in $\Rd$. The unique bounded solution of \eqref{p_e_intro} can be written in terms of an evolution operator $G(t,s)$ associated to $\A(t)$, i.e.,
 $$u(t,x)=(G(t,s)f)(x),\quad\;\, t>s,\, x\in \Rd.$$
 Many properties of the solution of problem \eqref{p_e_intro} are investigated in \cite{KunLorLun09Non}; in particular, in \cite[Sect. 4]{KunLorLun09Non} some sufficient conditions on the coefficients are provided in order that the pointwise gradient estimates
\begin{equation}
\label{grad_estimate_intro}
|(\nabla_x G(t,s)f)(x)|^p\leq e^{c_p(t-s)}(G(t,s)|\nabla f|^p)(x),\quad\;\, t>s,\,x\in \Rd,
\end{equation}
hold for every $p>1$, $f \in C^1_b(\Rd)$ and some $c_p\in \R$.

\noindent
 The interest in this kind of estimates is due to the fact that they play a crucial role in the analysis of many qualitative properties of $G(t,s)$. Already in the autonomous case, they have been used to study the asymptotic behavior of the semigroup $T(t)$ generated by the operator in \eqref{a(t)} (when $Q(t,x)=Q(x)$ and $b(t,x)=b(x)$) in $L^p(\Rd,\mu)$, where $\mu$ is an invariant measure of $T(t)$, i.e., a Borel probability measure such that $\int_{\Rd} T(t)f d\mu=\int_{\Rd} f d\mu$, for every $f\in C_b(\Rd)$ and any $t>0$. In fact, this is the case also in the nonautonomous setting, where $T(t)$ is replaced by $G(t,s)$ and the single invariant measure $\mu$ is replaced by a family of Borel probability measures $\{\mu_t:\,\,t\in I\}$ called \emph{evolution system of measures}, satisfying
 $$\int_{\Rd} (G(t,s)f)(x) d\mu_t(x)=\int_{\Rd} f(x) d\mu_s(x),\quad\;\, t>s\in I,\,f\in C_b(\Rd).$$
\noindent
In the case of $T$-time periodic (unbounded) coefficients, it has been proved in \cite{LorLunZam10} that, if the coefficients are smooth enough and a weak dissipativity condition on the drift $b$ is assumed, then
\begin{equation}\label{asy_beh}
\lim_{t \to +\infty} \|G(t,s)f-m_s(f)\|_{L^p(\Rd,\mu_t)}=0,\quad\;\, s\in \R,\,f\in L^p(\Rd,\mu_s),
\end{equation}
for every $p\in[1,+\infty)$, where
$$m_s (f)=\int_{\Rd}f(y)d\mu_s(y)$$
and $\{\mu_t:\,\,t\in \R\}$ is the $T$-periodic evolution system of measures.

\noindent
The asymptotic behavior stated in \eqref{asy_beh} still holds also in a non-periodic setting provided that estimate
\eqref{grad_estimate_intro} holds for $p=1$ and some $c_1<0$ (see \cite{AngLorLun}).
Hence the problem is reduced to find conditions that imply
\begin{equation}
\label{p=1}
|(\nabla_x G(t,s)f)(x)|\leq e^{c_1(t-s)}(G(t,s)|\nabla f|)(x),\quad\;\, t>s\in I, \,x\in \Rd,
\end{equation}
for functions $f \in C^1_b(\Rd)$.
This is the case (see \cite[Thm. 4.5]{KunLorLun09Non})
 if the coefficients $q_{ij}$ ($i,j=1,\dots,d$) do not depend on $x$ and
\begin{equation}\label{drift}
\langle \nabla_x b(t,x)\xi,\xi\rangle \leq r_0 |\xi|^2,\quad\;\, \xi \in \Rd,\,(t,x)\in I\times \Rd,
\end{equation}
for some $r_0 \in \R$. In this case, estimate \eqref{p=1} is satisfied with $c_1=r_0$.
\noindent
Actually, the gradient estimate \eqref{p=1} gives sharper information than formula \eqref{asy_beh}.
When it is satisfied (as it has been proved in \cite[Cor. 5.4]{AngLorLun}), the exponential decay estimate
\begin{equation}\label{decay}
\|G(t,s)f-m_s(f)\|_{L^p(\Rd,\mu_t)}\leq C_p e^{c_1(t-s)}\|f\|_{L^p(\Rd,\mu_s)},\quad\;\, t>s\in I,
\end{equation}
holds for every $p>1,\, f\in L^p(\Rd,\mu_s)$ and some $C_p>0$.\\
The fact that estimate \eqref{p=1} has been proved only when the diffusion coefficients do not depend on $x$ is not surprising since, already in the autonomous case, Wang (\cite{wang})
 proved that the gradient
estimate $|\nabla T(t)f|\leq e^{ct}T(t)|\nabla f|$
cannot hold
if the coefficients $q_{ij}$ do not satisfy the algebraic condition:
$$D_k q_{ij}(x)+D_i q_{kj}(x)+D_j q_{ki}(x)=0,\quad\;\, 1\leq i,j,k\leq d,\, x\in \Rd.$$

Estimate \eqref{p=1} has been also the key formula to establish many
other results on the summability improving properties of $G(t,s)$ in the $L^p$-spaces
related to the unique tight evolution system of measures $\{\mu_t:\,\,t\in I\}$.
In \cite{AngLorLun}, we use \eqref{p=1} in order to prove a Logarithmic-Sobolev
 inequality with respect to the tight system $\{\mu_t:\,\,t\in I\}$. Moreover, we establish a connection between the Logarithmic-Sobolev
 inequality and the hypercontractivity of the evolution operator $G(t,s)$ in the
 $L^p$-spaces related to the evolution system of measures $\{\mu_t:\,\,t\in I\}$.

In \cite{AngLor}, assuming \eqref{p=1}, we prove some Harnack type estimates and stronger results than hypercontractivity
for the evolution operator $G(t,s)$.

These results have been proved assuming
 that the diffusion coefficients do not depend on $x$ and formula \eqref{drift} is satisfied, so that \eqref{p=1} holds.

Because of the great importance of formula \eqref{p=1}, in this paper we provide two sufficient conditions on the coefficients $q_{ij}$ and $b_i$
in order that \eqref{p=1} is satisfied in the general case, and we show that one of them is also necessary.
More precisely we prove that, if the algebraic pointwise condition
\begin{equation}\label{intro_alge}
D_k q_{ij}(t,x)+D_i q_{kj}(t,x)+D_j q_{ik}(t,x)=0,\quad\;\,(t,x)\in I\times \Rd,
\end{equation}
is satisfied for every $i,j,k \in \{1,\dots,d\}$ and if the
dissipativity condition (which includes also the spatial derivatives of the diffusion coefficients $q_{ij}$)
$$\left(\frac{1}{2\eta(t,x)}\sum_{i,j=1}^d\langle \nabla_x q_{ij}(t,x),\xi\rangle^2\right)+\langle \nabla_x b(t,x)\xi,\xi\rangle\leq c_0|\xi|^2,
$$
holds for every $\xi\in \Rd$, $(t,x)\in I\times\Rd$ and some $c_0\in \R$, (see \eqref{ell} for the definition of $\eta$), then the gradient estimate \eqref{p=1} is satisfied. Moreover, as in the autonomous case, condition \eqref{intro_alge} is necessary for estimate \eqref{p=1}.

The proof of these facts relies on the connection between the gradient estimate \eqref{p=1} and the
 uniform Bakry type estimate
 \begin{equation}\label{BE_intro}
\langle \nabla f, \nabla_x (\mathcal{A}(s)f) \rangle\leq |\nabla f|(\mathcal{A}(s)|\nabla f|)+c|\nabla f|^2,\quad\;\, f\in C^\infty(\Rd),\,s\in I.
\end{equation}
Unfortunately, differently from the autonomous case (where they are equivalent, see \cite{Bak97OnS}), we are able to prove only that estimate \eqref{BE_intro}
is a necessary condition for the gradient estimate \eqref{p=1} hold, hence we prove the main result of the paper following a quite different approach than in \cite{wang}.

The paper is organized ad follows. In Section \ref{preliminary} we state our main assumptions,
we collect some known results on the evolution operator $G(t,s)$ and we prove a preliminary lemma.
Section \ref{main} contains a characterization of the occurrence of the gradient estimate \eqref{p=1}.
Finally, in Section \ref{comments} we give examples of nonautonomous operators to which the main result of the paper may be applied.

\subsection*{Notations}
Let $k\in [0,+\infty)$, we denote by $C^k_b(\Rd)$  the set of functions in $C^{[k]}(\Rd)$ which are bounded together with all their derivatives up to the $[k]$-th order and such that the $[k]$-th order derivatives are
$(k-[k])$-H\"older continuous in $\Rd$.
We use the subscript ``$c$'' instead of ``$b$''  for the subsets of the
above spaces consisting of functions  with compact support.

If $J \subset \R$ is an interval and $\alpha\in (0,1)$, $C^{k+\alpha/2,2k+\alpha}_{{\rm loc}}(J \times\Rd)$ ($k=0,1$)
denotes the set of functions $f:J\times \Rd\to \R$ such that the time derivatives up to the $k$-th order and the spatial derivatives up to the $2k$-th order are H\"older continuous with exponent $\alpha$, with respect to the parabolic distance, in any compact set of $J\times \Rd$.
Analogously we define the space of functions $C^{1+\alpha/2,3+\alpha}_{{\rm loc}}(J \times\Rd)$.

About partial derivatives, the notations $D_tf:=\frac{\partial f}{\partial t}$,
$D_if:=\frac{\partial f}{\partial x_i}$, $D_{ij}f:=\frac{\partial^2f}{\partial x_i\partial x_j}$ are extensively used.

About matrices and vectors, we denote by $\textrm{Tr}(Q)$, $\langle x,y\rangle$ and $|x|$ the trace of the square matrix
$Q$, the  inner product
of the vectors $x,y\in\Rd$ and the Euclidean norm of $x$, respectively.

The  ball in $\R^d$ centered at $x_0$ with  radius $r>0$ is denoted by $B(x_0,r)$.
When $x_0=0$, we simply write $B_r$ instead of $B(x_0,r)$.

\section{Assumptions, definitions and a review of some properties of $G(t,s)$}\label{preliminary}

First we state our standing assumptions and we collect some known results.

Let $I$ be an open right  halfline. For every $t\in I$, we consider
the linear second order differential
operator $ \mathcal{A}(t) $
defined on smooth functions $\zeta$ by
\begin{align*}
(\mathcal{A}(t)\zeta)(x)&=\sum_{i,j=1}^d q_{ij}(t,x)D_{ij}\zeta(x)+
\sum_{i=1}^d b_i(t,x)D_i\zeta(x)\\
&= \textrm{Tr}(Q(t,x)D^2\zeta(x))+ \langle b(t,x), \nabla \zeta(x)\rangle,\qquad\;\, x\in\R^d.
\end{align*}
The standing hypotheses on the data $Q=[q_{ij}]_{i, j=1, \ldots, d}$ and $b = (b_1, \ldots, b_d)$ are the following:

\begin{hyp}\label{hyp1}
\begin{enumerate}[\rm (i)]
\item
The coefficients $q_{ij}$ and $b_{i}$ $(i,j=1,\dots,d)$ and their first order spatial derivatives belong to $C^{\alpha/2,\alpha}_{\rm loc}(I\times \R^d)$  for
some $\alpha \in (0,1)$;
\item
the symmetric matrix $Q(t,x)=[q_{ij}(t,x)]_{i,j=1, \ldots, d}$
is uniformly elliptic, i.e., there exists a function $\eta:I\times \Rd\to \R$ such that $0<\eta_0=\inf_{I\times \Rd}\eta$ and
\begin{equation}
\label{ell}
\langle Q(t,x)\xi,\xi\rangle \ge \eta(t,x)|\xi|^2  ,\qquad\;\, \xi\in \Rd,\,\,(t,x)\in I\times\Rd;
\end{equation}
\item
for every bounded interval $J\subset I$ there exist a function $\varphi=\varphi_J\in C^2(\R^d)$ with positive values,  such that $\lim_{|x|\to +\infty}\varphi(x)=+\infty$, and a positive number $\gamma=\gamma_J$ such that
\begin{equation}
\label{Lyapunov}
(\mathcal{A}(t)\varphi)(x)\leq \gamma\,\varphi(x),  \quad (t,x)\in J\times \Rd.
\end{equation}
\end{enumerate}
\end{hyp}

Under these assumptions, for every $s\in I$ and $f\in C_b(\Rd)$, the problem
\begin{equation}\label{p_e}
\left\{
\begin{array}{ll}
D_tu(t,x)={\mathcal{A}}(t)u(t,x),\quad\quad & (t,x)\in (s,+\infty)\times \Rd,\\[1mm]
u(s,x)= f(x),\quad\quad & x\in \Rd,
\end{array}\right.
\end{equation}
admits a unique bounded classical solution, i.e., there exists a unique function
$u \in C_b([s,+\infty)\times \Rd)\cap C^{1,2}((s,+\infty)\times \Rd)$
that satisfies \eqref{p_e}. Moreover,
\begin{equation}\label{uniform_est}
\|u(t,\cdot)\|_{\infty}\leq \|f\|_{\infty},\quad\;\, t \geq s.
\end{equation}
We point out that condition (i) is not minimal for the well-posedness of problem \eqref{p_e}. In order to get existence and uniqueness of a solution to the problem \eqref{p_e}, besides Hypotheses \ref{hyp1}(ii)-(iii), it suffices to require only that the coefficients $q_{ij}$ and $b_i$ belong to $C^{\alpha/2,\alpha}_{\rm loc}(I\times \R^d)$.
The additional hypothesis on the regularity of the first-order spatial derivatives of the coefficients is used to prove that the solution is smoother.

The unique bounded solution $u$ to the problem \eqref{p_e} can be represented by means of a positive evolution operator $G(t,s)$ associated to $\A(t)$, by setting
$G(t,t):={\rm id}_{C_b(\Rd)}$ for every $t\in I$ and
$$(G(t,s)f)(x):= u(t,x),\quad\;\, (t,x)\in (s,+\infty)\times \Rd.$$

As already noticed, uniqueness of the solution of \eqref{p_e} is immediate consequence of Hypothesis \eqref{hyp1}(iii)
and is proved by means of the following maximum principle.
\begin{prop}\label{max_prin}
Let $s\in I$ and $T>s$. If $u \in C_b([s,T]\times \Rd)\cap C^{1,2}((s,T]\times \Rd)$ satisfies
\begin{equation*}
\left\{
\begin{array}{ll}
D_tu(t,x)-{\mathcal{A}}(t)u(t,x)\le 0,\quad\quad & (t,x)\in (s,T]\times \Rd,\\[1mm]
u(s,x)\le 0,\quad\quad & x\in \Rd,
\end{array}\right.
\end{equation*}
then $u(t,x)\le 0$ for every $(t,x)\in [s,T]\times \Rd$.
\end{prop}
\begin{proof}
See \cite[Thm. 2.1]{KunLorLun09Non} and the reference therein.
\end{proof}

The next lemma provides a regularity result when the initial datum $f$ is smooth enough.

\begin{lemm}\label{smoothdatum}
If $f\in C_c^{3+\alpha}(\Rd)$, then the solution $u$ to the problem \eqref{p_e} belongs to $C^{1+\alpha/2,3+\alpha}_{\rm loc}([s,+\infty)\times \Rd)$.
\end{lemm}
\begin{proof}
Assume that $f$ belongs to $C_c^{3+\alpha}(\Rd)$. Let $m$ be the smallest integer
such that ${\rm \supp}f\subset B_m$. For every $n>m$, we consider the Cauchy-Dirichlet problem
\begin{equation}\label{p_d_n}
\left\{
\begin{array}{ll}
D_tu(t,x)={\mathcal{A}}(t)u(t,x),\quad\quad & (t,x)\in (s,+\infty)\times B_n,\\[1mm]
u(s,x)= f(x),\quad\quad & x\in B_n,\\[1mm]
u(t,x)=0,\quad\quad & (t,x)\in (s,+\infty)\times \partial B_n.
\end{array}\right.
\end{equation}
From Hypothesis \ref{hyp1}(i) and the classical results in \cite[Thms. 3.3.7--3.5.12]{Fri64Par}, problem \eqref{p_d_n} admits a
 unique solution
$u_n \in C^{1+\alpha/2,3+\alpha}([s,+\infty)\times \overline{B}_n)$
 such that
\begin{equation}\label{norm_infty}
\|u_n(t,\cdot)\|_{\infty}\leq \|f\|_{\infty},\quad\:\, t\geq s.
\end{equation}
The local Schauder estimates (see \cite[Thm. IV.10.1]{LadSolUra68Lin}) and estimate \eqref{norm_infty} yield that, for every $k<n$, there exists a positive constant $c_k$, independent on $n$, such that
\begin{align*}
\| u_n\|_{C^{1+\alpha/2,3+\alpha}([s, s+k]\times B_k)}\leq  c_k \| f\|_{C^{3+\alpha}_b(\Rd)} .
\end{align*}
By the Arzel\`a-Ascoli theorem we deduce that there exists a subsequence $(u_n^k)$ of $(u_n)$ which converges in
$C^{1,3}([s,s+k]\times \overline{B}_k)$ to a function $u^k\in C^{1+\alpha/2, 3+\alpha}([s,s+k]\times \overline{B}_k)$,
which satisfies the equation $u^k_t=\A(t)u^k$ in $(s,s+k)\times B_k$. Moreover, $u^k(s,\cdot)= f$ in $B_k$. Since, without loss of generality, we can assume that $(u_n^{k+1})$ is a subsequence of $(u_n^{k})$ and hence $u^{k+1}=u^{k}$ in $(s,s+k)\times B_k$, we can define the function $u:[s,+\infty)\times\Rd\to\R$ by setting $u(t,x)=u^k(t,x)$ for every $(t,x)\in(s, s+k)\times B_k$ and every $k \in \N$. The function $u$ belongs to $C^{1+\alpha/2, 3+\alpha}_{\rm loc}([s,+\infty)\times\Rd)$, satisfies \eqref{norm_infty} and it is the unique solution of problem \eqref{p_e}, due to Proposition \ref{max_prin}.
\end{proof}

In the next proposition, following the ideas in \cite{Bak97OnS}, we establish a connection between the gradient estimate satisfied by $G(t,s)$  and the
Bakry type estimate \eqref{BE_intro} (introduced in the autonomous setting in \cite{Bak97OnS})
satisfied by the operator $\A(t)$. More precisely, we prove that the Bakry type estimate is a
necessary condition for the gradient estimate \eqref{grad_estimate_intro} hold.

\begin{prop}\label{lemma_notcorrect}
Assume that there exists $c\in \R$ such that, for every $f\in C^1_b(\Rd)$ and $I\ni s\leq t$,
\begin{equation}\label{GE}
|\nabla_x G(t,s)f| \leq e^{c(t-s)}G(t,s)|\nabla f|.
\end{equation}
Then, the estimate
\begin{equation}\label{BE}
\langle \nabla f, \nabla_x (\mathcal{A}(s)f) \rangle\leq |\nabla f|\mathcal{A}(s)|\nabla f|+c|\nabla f|^2
\end{equation}
holds for every $f\in C^3(\Rd)$ and $s\in I$.
\end{prop}
\begin{proof}
It suffices to prove
\eqref{BE} at any $x_0\in \Rd$ such that $|\nabla f(x_0)|>0$.
Formula \eqref{GE} yields
\begin{equation}\label{der_est}
\frac{1}{t-s}\left(|\nabla_x G(t,s)f|^2-|\nabla f|^2\right)
\leq \frac{1}{t-s}\left(e^{2c(t-s)}\big(G(t,s)|\nabla f|)^2-|\nabla f|^2\right),
\end{equation}
for any $t>s\in I$. We notice that the left and the right hand sides of \eqref{der_est} represent, respectively,
the incremental ratio at $t=s$ of the functions
$t \mapsto |\nabla_x G(t,s)f|^2=:h_{1}(t)$ and $t \mapsto e^{2c(t-s)}(G(t,s)|\nabla f|)^2=:h_2(t)$.

We prove first \eqref{BE} for $f\in C^\infty_c(\Rd)$.
The smoothness of the coefficients $q_{ij}$ and $b_i$ and Lemma \ref{smoothdatum} yield that the first-order spatial derivatives of $G(t,s)f$ belong to $C^{1+\alpha/2,2+\alpha}_{{\rm loc}}((s,+\infty)\times \Rd)$, hence
$$D_t(\nabla_x G(t,s)f)=\nabla_x(D_t (G(t,s)f))=\nabla_x (\mathcal{A}(t)G(t,s)f)$$
and consequently
$$h_1'(t)=2 \langle \nabla_x (\A(t)G(t,s)f),\nabla_x G(t,s)f \rangle,\quad\;\, t>s.$$
Moreover, again the smoothness of $f$ and of the coefficients of $\A(t)$, together with Lemma \ref{smoothdatum}, imply that the functions $\nabla_x G(\cdot,s)f$ and
$\nabla_x (\A(\cdot)G(\cdot,s)f)$ are continuous in $[s,+\infty)\times \Rd$. Hence $h_1$ is differentiable also in $t=s$ and $h_1'(s)=2 \langle \nabla_x (\A(s)f),\nabla f \rangle$.
Let us observe that the derivative of the function $h_2$ is given by
\begin{equation*}\label{derivative}h'_2(t)= 2c h_2(t)+2e^{2c(t-s)}(G(t,s)|\nabla f|)(\A(t)G(t,s)|\nabla f|),\quad\;\, t>s.
\end{equation*}
Since the function $t \mapsto \A(t)G(t,s)|\nabla f|$ is not (necessarily) continuous up to $s$, we consider a function
\begin{equation}\label{g_function}
g\in C^\infty_c(\Rd),\quad g=|\nabla f|\,\, \textrm{in a neighborhood of}\,\, x_0 \,\,\,\,\textrm{and}\quad g \ge |\nabla f| \,\, \textrm{in}\,\, \Rd.
\end{equation}
In this case, $G(\cdot,s)g\in C^{1+\alpha/2,2+\alpha}_{{\rm loc}}([s,+\infty)\times \Rd)$ and $(\mathcal{A}(s)g)(x_0)=(\mathcal{A}(s)|\nabla f|)(x_0)$. From \eqref{GE}, \eqref{g_function} and the positivity of $G(t,s)$ we deduce that
\begin{equation}\label{est_g}
|\nabla_x G(t,s)f|^2\le e^{2c(t-s)}\big(G(t,s)g)^2,
\end{equation}
with equality at $t=s$.
Taking the derivatives with respect to $t$ at $t=s$ of both sides in \eqref{est_g}, we get
\begin{equation*}
2 \langle \nabla_x (\A(s)f),\nabla f \rangle \le 2\big(c g^2+g(\A(s)g)\big),
\end{equation*}
hence,
$$\langle \nabla_x (\A(s)f)(x_0),\nabla f(x_0) \rangle \le c |\nabla f(x_0)|^2+|\nabla f(x_0)|(\A(s)|\nabla f|)(x_0),\quad\; s\in I.$$
To conclude the proof in this case, we determine a function $g$ which satisfies \eqref{g_function}.
Let $r>0$ be such that $|\nabla f(y)|>0$ for $|y-x_0|\leq r$. Let us consider two functions $\theta,\psi \in C^\infty_c(\Rd)$ such that $\theta=1$ in $B(x_0,r/2)$, $\theta=0$ in $\Rd\setminus B(x_0,r)$ and $\psi=1$ in the support of $f$. Then, the function
$$g(y):=\psi(y)[\theta(y)|\nabla f(y)|+(1-\theta(y))\|\nabla f\|_{\infty}],\quad\;\, y\in \Rd,$$
satisfiess all the properties claimed in \eqref{g_function}.
By the arbitrariness of $x_0\in \Rd$ we get \eqref{BE} for any function $f \in C^\infty_c(\Rd)$.

Finally, if $f\in C^3(\Rd)$ we can consider a sequence of functions $f_n\in C^\infty_c(\Rd)$, which converges locally uniformly to $f$, and
the sequence of functions $\tilde{f_n}:=\theta_n f_n$, where $\theta_n$ is defined as follows
\begin{equation}
\label{thetan}
\theta_n(x)=\psi\left(\frac{|x|}{n}\right), \qquad\;\,x\in \R^d, \;\,n\in \N,
\end{equation}
and
$\psi \in C^\infty(\R)$ satisfies $\chi_{(-\infty,1]}\leq \psi \leq \chi_{(-\infty,2]}$.
Then, $\tilde{f_n}\in C^\infty_c(\Rd)$ for every $n \in \N$ and $(D^{|\alpha|}f_n)(x)$ converges to $(D^{|\alpha|}f)(x)$ as $n\to +\infty$ for every $x\in \Rd$ and $0\leq |\alpha|\leq 3$.
Hence, writing \eqref{BE} for $\tilde{f_n}$ and letting $n\to +\infty$ we get the claim.
\end{proof}

\section{Main theorem}\label{main}

This section is devoted to prove the main result of the paper. In the following theorem some sufficient
conditions in order that the pointwise gradient estimate \eqref{p=1} hold are given.

\begin{thm}\label{mainthm}
Assume that, for every $i,j,k=1,\dots,d$,
\begin{equation}\label{cond_algebrica}
D_k q_{ij}(t,x)+D_i q_{kj}(t,x)+D_j q_{ik}(t,x)=0,\quad\;\,(t,x)\in I\times \Rd,
\end{equation}
and that there exists  $c_0\in \R$ such that
\begin{equation}
\label{b}
\left(\frac{1}{2\eta(t,x)}\sum_{i,j=1}^d\langle \nabla_x q_{ij}(t,x),\xi\rangle^2\right)+\langle \nabla_x b(t,x)\xi,\xi\rangle\leq c_0|\xi|^2,
\end{equation}
for every $\xi\in \Rd$ and $(t,x)\in I\times\Rd$, where $\eta$ is the function defined in \eqref{ell}. Then, for every $f\in C^1_b(\Rd)$ and $I\ni s\leq t$,
\begin{equation}\label{c}
|(\nabla_x G(t,s)f)(x)| \leq e^{c_0(t-s)}(G(t,s)|\nabla f|)(x),\quad\;\, x\in \Rd.
\end{equation}
Conversely, assume that the gradient estimate \eqref{c} is satisfied for some $c_0\in \R$. Then \eqref{cond_algebrica} holds for every $t\in I$, $x\in \Rd$ and $i,j,k=1,\dots,d$.
\end{thm}

\begin{proof}

We prove the first part of the statement by using a variant of the Bernstein method. Fix $s\in I$ and $\varepsilon>0$.
For every $f\in C^1_b(\Rd)$, set $u(t,x):=(G(t,s)f)(x)$ and define
$$w(t,x)=(|\nabla_x u(t,x)|^2+\varepsilon)^{1/2},\quad\, t \ge s,\,\,x \in \Rd.$$
By \cite[Thm. 3.10]{Fri64Par} and \cite[Thm. 4.1, Cor. 4.4]{KunLorLun09Non}, $w\in C_b([s,T]\times \Rd)\cap C^{1,2}((s,T)\times \Rd)$ for every $T>s$ and
a straightforward computation shows that
$$D_t w-\A(t)w=F,$$
where
\begin{align*}
F&= (|\nabla_x u|^2+\varepsilon)^{-1/2}\Big(\langle \nabla_x b \nabla_x u,\nabla _x u\rangle-\sum_{k=1}^d\langle Q\nabla_x D_ku,\nabla_x D_k u\rangle\Big),\\
&\quad\quad\quad\quad\quad+(|\nabla_x u|^2+\varepsilon)^{-1/2}\sum_{k=1}^d D_k u\cdot {\rm Tr}(D_k Q\cdot D^2_x u)\\
&\quad\quad\quad\quad\quad+(|\nabla_x u|^2+\varepsilon)^{-3/2}\langle Q D^2_x u\nabla_x u, D^2_x u \nabla_x u\rangle.
\end{align*}
First of all, let us observe that
\begin{align*}
F\leq (|\nabla_x u|^2&+\varepsilon)^{-1/2}\left(\langle \nabla_x b \nabla_x u,\nabla _x u\rangle-\sum_{k=1}^d\langle Q\nabla_x D_ku,\nabla_x D_k u\rangle\right.\\
&\quad+\left.\sum_{k=1}^d D_k u\cdot {\rm Tr}(D_k Q\cdot D^2_x u)+\left\langle Q D^2_x u\frac{\nabla_x u}{|\nabla_x u|}, D^2_x u \frac{\nabla_x u}{|\nabla_x u|}\right\rangle\right).
\end{align*}
Moreover,
\begin{align}\label{usef_est}&\left\langle Q D^2_x u\frac{\nabla_x u}{|\nabla_x u|}, D^2_x u \frac{\nabla_x u}{|\nabla_x u|}\right\rangle-\sum_{k=1}^d\langle Q\nabla_x D_ku,\nabla_x D_k u\rangle\nnm\\
= & \sum_{i,j=1}^d q_{ij}\left(\frac{\langle \nabla_x u, \nabla_x D_i u \rangle\langle \nabla_x u, \nabla_x D_j u \rangle}{|\nabla_x u|^2} -\langle \nabla_x D_i u,\nabla_x D_j u\rangle \right)\nnm\\
= & -\sum_{i,j=1}^d q_{ij}\langle P(\nabla_x D_i u),P(\nabla_x D_j u)\rangle\leq -\eta\sum_{i=1}^d|P(\nabla_x D_i u)|^2,
\end{align}
where $P$ denotes the projection
\begin{equation}\label{projection}
P(v)= v-\left\langle v, \frac{\nabla_x u}{|\nabla_x u|}\right\rangle\frac{\nabla_x u}{|\nabla_x u|}, \qquad\;\, v \in \Rd.
\end{equation}
Hence, we have
\begin{align}\label{def_I}
F\leq &\frac{1}{w}\left(\langle \nabla_x b \nabla_x u,\nabla _x u\rangle-\sum_{i,j=1}^d q_{ij}\langle P(\nabla_x D_i u),P(\nabla_x D_j u)\rangle\right.\nonumber\\
&\quad\quad\quad\quad\quad+\left.\sum_{k=1}^d D_k u\cdot {\rm Tr}(D_k Q\cdot D^2_x u)\right)=: \displaystyle{\frac{1}{w}\,\,I.}
\end{align}
The crucial point of the first part of the proof consists in proving that
\begin{equation}\label{aim}
I(t,x)\leq c_0 |\nabla_x u(t,x)|^2 ,
\end{equation}
for every $t>s$ and $x \in \Rd$, where $c_0$ is the constant in assumption \eqref{b}. Indeed, in this case we obtain $D_t w-\mathcal{A}(t)w\le c_0w$.
Since, on the other hand, the function
$$z(t,\cdot)= e^{c_0(t-s)}G(t,s)(|\nabla f|^2+\varepsilon)^{\frac{1}{2}},\qquad\;\, t>s,$$
satisfies $D_t z-\mathcal{A}(t)z= c_0 z$, we get
\begin{eqnarray*}
\left\{
\begin{array}{ll}
D_t(w-z)(t,x)-\left[({\mathcal{A}}(t)+c_0)(w-z)\right](t,x)\le 0,\quad\quad & (t,x)\in (s,+\infty)\times \Rd,\\[1mm]
(w-z)(s,x)=0,\quad\quad & x\in \Rd.
\end{array}\right.
\end{eqnarray*}
Thus, the maximum principle in Proposition \ref{max_prin} implies that $w \le z$. Letting $\varepsilon \to 0^+$ and using the continuity property of $G(t,s)$ that follows from estimate \eqref{uniform_est}, we get \eqref{c}.

Now, let us fix $x_0\in \Rd$ and $t>s$ and prove that $I(t,x_0)\leq c_0 |\nabla_x u(t,x_0)|^2$ . We point out that it is not restrictive, from now on, to assume that the coefficients $q_{ij}$ are linear functions. Indeed, if we denote by $\widetilde{I}$ the sum in brackets in formula \eqref{def_I} where the $q_{ij}$'s are replaced by the $\widetilde{q}_{ij}$'s, defined by $\widetilde{q}_{ij}(t,x)=q_{ij}(t,x_0)+\langle \nabla_x q_{ij}(t,x_0),x-x_0\rangle$, ($i,j=1,\dots,d$),  we notice that $\widetilde{I}(t,x_0)=I(t,x_0)$. Moreover $\widetilde{q}_{ij}$ and $b_i$ satisfy the assumptions \eqref{cond_algebrica} and \eqref{b} at $(t,x_0)$ with the same constant $c_0$, and this is enough to complete the proof.

\noindent
We have
\begin{equation}
\sum_{k=1}^d D_k u\cdot {\rm Tr}(D_k Q\cdot D^2_x u)= \sum_{j=1}^d\langle \nabla_x D_j u, Q^j \rangle,
\end{equation}
where, for every $j=1,\dots, d$ and $(t,x)\in I\times \Rd$, $Q^j(t,x)$ is the vector with components $Q^j_i(t,x)=\sum_{k=1}^d D_k u(t,x) D_k q_{ij}(t,x)$
for $1\le i\le d$. Taking into account the definition of $P$ in \eqref{projection}, we can write
$$\langle \nabla_x D_j u, Q^j \rangle= \langle P(\nabla_x D_j u),P( Q^j) \rangle+\frac{1}{|\nabla_x u|^2}\langle \nabla_x u, \nabla_x D_j u  \rangle\langle \nabla_x  u, Q^j \rangle.$$
Moreover, being
$$
\langle \nabla_x u, \nabla_x D_j u \rangle\langle \nabla_x  u, Q^j \rangle =\langle \nabla_x u, \nabla_x (D_j u\langle\nabla_x u,Q^j\rangle)\rangle-D_j u \langle \nabla_x u, \nabla_x\langle \nabla_x u,Q^j\rangle\rangle
$$
and $\sum_{j=1}^d D_j u\langle\nabla_x u,Q^j\rangle=0$ by the assumption \eqref{cond_algebrica}, we get
\begin{align*}
\sum_{j=1}^d \langle \nabla_x u, &\nabla_x D_j u \rangle\langle \nabla_x  u, Q^j \rangle
= -
\sum_{j=1}^d D_j u \langle \nabla_x u, \nabla_x\langle \nabla_x u,Q^j\rangle\rangle\\
& =-\sum_{i,j,k,l=1}^d D_j u D_k u(D_i u D_{kl} u D_l q_{ij}+D_l u D_{ik} u D_l q_{ij})\\
& = -\sum_{i,j,k,l=1}^d D_j u D_k u\left[D_i u D_{kl} u D_l q_{ij}+D_l u D_{ik} u (-D_i q_{lj}-D_j q_{il})\right]\\
& = \sum_{i,j,k,l=1}^d D_j u D_k u D_l u D_{ik} u D_j q_{il},
\end{align*}
where we have used the linearity of $q_{ij}$ and again assumption \eqref{cond_algebrica}. Then,
$$\sum_{j=1}^d \langle \nabla_x u,\nabla_x D_j u\rangle\langle \nabla_x  u, Q^j \rangle=\sum_{k=1}^d \langle R^k,\nabla_x D_k u \rangle,$$
where,  for every fixed $k=1, \dots, d$ and $(t,x)\in I\times \Rd$, $R^k(t,x)$ denotes the vector with components $R^k_i(t,x)=D_k u(t,x)\sum_{j,l=1}^dD_j u(t,x) D_l u(t,x) D_j q_{il}(t,x)$ for $1\le i\le d$.
Finally, since assumption \eqref{cond_algebrica} implies $\langle R^k,\nabla_x u\rangle=0$, we have
$$\sum_{j=1}^d \langle \nabla_x u ,\nabla_x D_j u\rangle\langle \nabla_x  u, Q^j \rangle=\sum_{k=1}^d \langle R^k,P(\nabla_x D_k u) \rangle.$$
Putting together all these results, we deduce
\begin{eqnarray*}
\sum_{k=1}^d D_k u\cdot {\rm Tr}(D_k Q\cdot D^2_x u)=\sum_{j=1}^d \Big\langle P(\nabla_x D_j u),P(Q^j)+\frac{|R^j|}{|\nabla_x u|^2} \Big\rangle.
\end{eqnarray*}
The Cauchy-Schwarz and the Young inequalities yield that
\begin{align}\label{est_mix}
\sum_{k=1}^d D_k u\cdot {\rm Tr}(D_k Q\cdot D^2_x u)&\leq\sum_{j=1}^d |P(\nabla_x D_j u)|\left(|P(Q^j)|+\frac{|R^j|}{|\nabla_x u|^2}\right)\nnm\\
& \leq \sum_{j=1}^d |P(\nabla_x D_j u)|\left(2|P(Q^j)|^2+2\frac{|R^j|^2}{|\nabla_x u|^4}\right)^{\frac{1}{2}}\nnm\\
& \leq \left(\sum_{j=1}^d |P(\nabla_x D_j u)|^2\right)^{\frac{1}{2}}\left[2\sum_{j=1}^d\left(|P(Q^j)|^2+\frac{|R^j|^2}{|\nabla_x u|^4}\right)\right]^{\frac{1}{2}}\nnm\\
& \leq \varepsilon \sum_{j=1}^d |P(\nabla_x D_j u)|^2+\frac{1}{2\varepsilon}\sum_{j=1}^d\left(|P(Q^j)|^2+\frac{|R^j|^2}{|\nabla_x u|^4}\right).
\end{align}
Choosing $\varepsilon=\eta(t,x_0)$ in \eqref{est_mix} and using \eqref{usef_est}, we get
\begin{align*}
I(t,x_0) \le \frac{1}{2\eta(t,x_0)}\sum_{j=1}^d\left(|P(Q^j)|^2+\frac{|R^j|^2}{|\nabla_x u|^4}\right)+\langle \nabla_x b \nabla_x u,\nabla _x u\rangle.
\end{align*}
Now, since
$$|P(Q^j)|^2= |Q^j|^2-\left\langle Q^j,\frac{\nabla_x u}{|\nabla_x u|}\right\rangle^2$$ and
$$\sum_{j=1}^d |R^{j}|^2=|\nabla_x u|^2\sum_{j=1}^d\langle Q^j, \nabla_x u\rangle^2,$$
we conclude that
$$I(t,x_0)\leq \frac{1}{2\eta(t,x_0)}\sum_{j=1}^d |Q^j|^2+\langle \nabla_x b \nabla_x u,\nabla _x u\rangle.$$
Finally, being
$$\sum_{j=1}^d |Q^j|^2= \sum_{i,j=1}^d  (Q^j_i)^2=\sum_{i,j=1}^d\left(\sum_{k=1}^d D_k u D_k q_{ij}\right)^2=\sum_{i,j=1}^d\langle \nabla_x u,\nabla_x q_{ij} \rangle^2, $$
by assumption \eqref{b}, we deduce that $I(t,x_0)\leq c_0 |\nabla_x u(t,x_0)|^2$ as claimed.

The second part of the statement can be obtained arguing as in \cite[Thm. 1.1(1)]{wang}
but, for the readers convenience, we give a sketch of the proof.

Let us assume that estimate \eqref{c} holds for some $c_0\in \R$. Then, Proposition \ref{lemma_notcorrect}
implies that estimate \eqref{BE} is satisfied too. We show how, throughout a suitable choice of smooth functions $f$, formula \eqref{BE} implies \eqref{cond_algebrica} in the three cases, respectively $i=j=k$, $i\neq j$ with $k\in\{i,j\}$ and $i\neq j$ with $k\notin\{i,j\}$.

Fix $t\in I$, $x\in \Rd$; let us assume that $i=j=k$ and consider the function $f$ defined by $f(y)=\cos(y_i-x_i)$ for any $y\in \Rd$;
from \eqref{BE}, for every $t\in I$, $y\in \Rd$ and $\varepsilon>0$ small enough, we get
\begin{equation}\label{first}
\left\{\begin{array}{ll}
D_i q_{ii}(t,y)\leq (c_0-D_i b_i(t,y))\tan (y_i-x_i),\quad\;\, &y_i-x_i\in (0,\varepsilon),\\[2mm]
D_i q_{ii}(t,y)\geq (c_0-D_i b_i(t,y))\tan (y_i-x_i),\quad\;\, &y_i-x_i\in (-\varepsilon,0).
\end{array}\right.
\end{equation}
Hence, letting $y\to x$ in the inequalities \eqref{first} we get $D_iq_{ii}(t,x)=0$, so that \eqref{cond_algebrica} holds.

\noindent
In the second case, if, for instance, $i\neq j$ and $k=i$, we have to prove that $2D_i q_{ij}(t,x)+D_j q_{ii}(t,x)=0$.
For every $\varepsilon>0$, let us consider the function $f$ defined by $f(y)=[\varepsilon (y_j-x_j)+(y_i-x_i)]^2$ for any $y\in \Rd$. From \eqref{BE}, taking into account that, by the previous step,
$D_k q_{kk}(t,x)=0$ for every $(t,x)\in I\times\Rd$ and $k=1,\dots,d$, we get that, if $y_j-x_j>0$ and $y_i-x_i>0$, then
\begin{align}\label{first_second}
2D_i q_{ij}(t,y)+D_j q_{ii}(t,y)\leq &-\varepsilon(2D_j q_{ij}(t,y)+D_i q_{jj}(t,y))\nnm\\
&\quad +\frac{\varepsilon (y_j-x_j)+(y_i-x_i)}{\varepsilon} \psi_\varepsilon(t,y),
\end{align}
where
$$\psi_\varepsilon(t,y)=\left[c_0(1+\varepsilon^2)-(\varepsilon^2 D_jb_j+\varepsilon(D_jb_i+D_ib_j)+D_ib_i)\right].$$
Analogously, if $y_j-x_j<0$ and $y_i-x_i<0$, we get the inverse inequality of \eqref{first_second}.
Therefore, letting first $y\to x$  and then $\varepsilon \to 0^+$ in both of the obtained inequalities, we get $2D_i q_{ij}(t,x)+D_j q_{ii}(t,x)=0$.

\noindent
In the last case, if $i\neq j$ and $k\notin\{i,j\}$, we consider the function $f$ defined by $f(y)=[(y_k-x_k)+(y_i-x_i)+(y_j-x_j)]^2$ for any $y\in \Rd$.
Using again \eqref{BE}, the results obtained in the previous two cases and arguing as before
(distinguishing the two cases $y_l-x_l>0$ and $y_l-x_l<0$ ($l\in\{i,j,k\}$)), we deduce that
$D_k q_{ij}(t,x)+D_i q_{kj}(t,x)+D_j q_{ki}(t,x)=0$, and the proof is now complete.
\end{proof}

\section{Comments and examples}\label{comments}

In \cite[Thm. 4.5]{KunLorLun09Non}, estimate \eqref{GE} has been proved when the diffusion coefficients of $\A(t)$ do not depend on $x$.
In this section we provide concrete examples of nonautonomous operators like \eqref{a(t)} whose diffusion matrices depend also on $x$ and whose associated evolution operators $G(t,s)$ satisfy the gradient estimate \eqref{GE}.

First, in the following remark we point out that, in some simple case, the algebraic condition \eqref{cond_algebrica}
forces the diffusion matrix to be independent of $x$, coming back trivially to the case considered in  \cite{KunLorLun09Non}.

\begin{rmk}\rm{
\begin{enumerate}[\rm (i)]
\item Let us consider the nonautonomous operator \eqref{a(t)} whose diffusion matrix $Q(t,x)$ is of the form $q(x)H(t)$ where $q\in C^{1+\alpha}_{\rm{loc}}(\Rd)$ and $H(t)=[h_{ij}(t)]_{i,j=1,\dots,d}$ has entries $h_{ij}\in C^{\alpha/2}_{\rm{loc}}(I)$ for every $i,j=1,\dots,d$. If \eqref{cond_algebrica} is satisfied, then $q(x)=c$ for every $x\in \Rd$ and some $c\in \R$. To check this fact, it suffices to write \eqref{cond_algebrica} for $i=j=k\in\{1,\dots,d\}$.
\item Assume that the matrix $Q(t,x)=[q_{ij}(t,x)]_{i,j=1,\dots,d}$ in \eqref{a(t)} is such that $q_{ij}(t,x)=a_i(t,x)\delta_{ij}$ for every $i,j=1,\dots,d$. If \eqref{cond_algebrica} is assumed to hold, then $Q(t,x)=Q(t)$; indeed, if $i=j\neq k$ formula \eqref{cond_algebrica} yields $D_k a_i(t,x)=0$ for every $k\neq i$, moreover, if $i=j=k$ we also deduce that $D_i a_i(t,x)=0$.
\end{enumerate}}
\end{rmk}
\noindent
Now, we exhibit some class of nonautonomous operators whose diffusion coefficients depend on the space variable $x$ and
to which the result in Theorem \ref{mainthm} may be applied.
\begin{example}\label{ex_1}\rm{
Consider the class of nonautonomous elliptic operators defined on smooth functions $\zeta$ by
$$(\A(t)\zeta)(x)= {\rm{Tr}}(Q(t,x)D^2\zeta(x))+\langle b(t,x), \nabla \zeta(x)\rangle,\quad\;\, t\in I,\, x\in \R^3.$$
Here,
\begin{equation}\label{Q_matrix_11}
Q(t,x_1,x_2,x_3)=\begin{pmatrix}
a_1(t)+\psi(t)x_2^2 & -\psi(t)x_1x_2&0\\[2mm]
-\psi(t)x_1 x_2 & a_2(t)+\psi(t)x_1^2 & 0\\[2mm]
0 &0 & a_3(t)
\end{pmatrix}
\end{equation}
\vskip 1mm
and
\begin{equation*}
b(t,x)=-\gamma(t)x|x|^{2\beta},\quad\;\, \beta \in [1,+\infty).
\end{equation*}
The positive functions $a_i,\psi,\gamma$ satisfy the following conditions:
\begin{enumerate}[\rm (i)]
\item $a_i,\psi,\gamma\in C^{\alpha/2}_{\textrm{loc}}(I)$ for $i=1,2,3$;
\item $\inf_{t \in I} a_i(t)>0$ for $i=1,2,3$;
\item $\displaystyle{\gamma>\max\left\{\bar{a},\psi,2\psi^2/\bar{a}\right\}}$ where $\bar{a}(t):=\min_{i=1,2,3}\{a_i(t)\}$.
\end{enumerate}
As it can be easily seen, $Q(t,x)$ is a positive definite matrix for any $(t,x)\in I\times \R^3$ and satisfies the condition \eqref{cond_algebrica}.
Moreover, the function
\begin{eqnarray*}
\varphi(x)=1+|x|^{2},\quad\;\, x\in \R^3,
\end{eqnarray*}
satisfies Hypothesis \ref{hyp1}(iii).
Indeed, for $t\in I$ and $x\in \R^3$ we have
%
\begin{align*}
(\mathcal{A}(t)\varphi)(x)=&  2 \big[ (\textrm{Tr}(Q(t,x))+  \langle b(t,x),x\rangle \big]\\
\leq &  2\big[a_1(t)+a_2(t)+a_3(t)+\psi(t)|x|^2-\gamma(t)|x|^{2(\beta+1)}\big].
\end{align*}
Hence,
$$\left(\frac{\A(t)\varphi}{\varphi}\right)(x)\longrightarrow -\infty,\quad\;\, \textrm{as}\,\, |x|\to +\infty, $$
uniformly with respect to $t\in I$, Thus, formula \eqref{Lyapunov} is satisfied.

Finally, we prove estimate \eqref{b}.
Let us observe that the matrix $Q$ is the sum of a semi-definite matrix and of a diagonal matrix whose diagonal
elements are respectively $a_1(t), a_2(t)$ and $a_3(t)$, hence the function $\eta$ in Hypothesis \ref{hyp1}(ii) is such that $\eta(t,x)\ge\bar{a}(t)$ for any $t\in I$ and $x\in \R^3$.
Moreover,
\begin{align*}
\sum_{i,j=1}^3\langle \nabla_x q_{ij}(t,x),\xi\rangle^2&=2\psi^2(t)[2x_2^2\xi_2^2+x_2^2\xi_1^2+x_1^2\xi_2^2+2x_1^2\xi_1^2+2x_1x_2\xi_1\xi_2]\\[-3mm]
&\leq 2\psi^2(t)[2x_2^2\xi_2^2+2x_2^2\xi_1^2+2x_1^2\xi_2^2+2x_1^2\xi_1^2]\\
& \leq 4 \psi^2(t) |x|^2|\xi|^2
\end{align*}
and
\begin{equation*}
\langle \nabla_x b(t,x)\xi,\xi \rangle= -\gamma(t)|x|^{2\beta}|\xi|^2-2\beta \gamma(t)|x|^{2(\beta-1)}\langle x,\xi\rangle^2,
\end{equation*}
for any $t\in I$ and $x,\xi\in \R^3$.
Therefore, we get
\begin{align*}
\frac{1}{2\eta(t,x)}\sum_{i,j=1}^3\langle \nabla_x q_{ij}(t,x),\xi\rangle^2+ \langle \nabla_x b(t,x)\xi,\xi \rangle  &\leq \left( \frac{2 \psi^2(t)}{\bar{a}(t)}|x|^2-\gamma(t)|x|^{2\beta}\right)|\xi|^2\\
&=: c(t,x)|\xi|^2.
\end{align*}
Since, as $|x|\to+\infty$, the function $c$ tends to $-\infty$ uniformly with respect to $t\in I$, we can conclude that
there exist a constant $c_0\in \R$ such that $c(t,x)\leq c_0$ for every $t\in I$ and $x\in \R^3$. Hence, \eqref{b} holds.
}\end{example}

\begin{rmk}\rm{
The Example \ref{ex_1} can be extended to the $d$-dimensional case. Indeed, we can consider a block diagonal matrix of the form
\begin{equation*}
Q(t,x)=\begin{pmatrix}
Q_{1}(t,x) &0           &\dots     &\dots&0\\
0          &\ddots        &\ddots   &\dots &\vdots\\
\vdots          &\ddots         &Q_{i}(t,x) &\ddots &\vdots    \\
\vdots          &\vdots        &\ddots   &\ddots    &0\\
0         &\dots           &\dots &0&Q_{k}(t,x)
\end{pmatrix}, \quad (t,x)\in I\times \Rd,
\end{equation*}
where each block $Q_{i}$ is either a three-dimensional matrix of the form of $Q$ in \eqref{Q_matrix_11}
or a two-dimensional matrix of the form
\begin{equation*}
Q(t,x,y)=\begin{pmatrix}
a_1(t)+\psi(t)y^2 & -\psi(t)xy\\[3mm]
-\psi(t)xy & a_2(t)+\psi(t)x^2
\end{pmatrix},
\end{equation*}
and the functions $a_1,a_2$ and $\psi$ satisfy conditions (i),(ii) and (iii) in Example \ref{ex_1}.
Moreover, for every $i=1,\dots,k$,
$$Q_i(t,x)= Q_{i}(t,x_{n_{i-1}+1},\dots, x_{n_{i-1}+n_i}),$$
where $n_0=0$ and $n_i\in\{2,3\}$ denotes the dimension of the block $Q_i$.
}
\end{rmk}

\end{document}